\documentclass[a4paper]{amsart}
\usepackage[T1]{fontenc}
\usepackage{cite}
\usepackage{sansmathfonts}

\usepackage[english]{babel}
\usepackage{amsmath}
\usepackage{amsfonts}
\usepackage{enumerate}
\usepackage[colorinlistoftodos,prependcaption,textsize=tiny]{todonotes}
\usepackage{amssymb}
\usepackage{amsthm}
\newtheorem{Thm}{Theorem}[section]
\newtheorem{Lmm}[Thm]{Lemma}
\newtheorem{Prp}[Thm]{Proposition}

\newtheorem{Crl}[Thm]{Corollary}

\theoremstyle{definition}
\newtheorem{Dfn}[Thm]{Definition}

\theoremstyle{remark}
\newtheorem{Rmk}[Thm]{Remark}

\usepackage[all]{xy}

\usepackage{mathrsfs}

\newcommand{\C}{\mathbb{C}}
\newcommand{\N}{\mathbb{N}}
\newcommand{\R}{\mathbb{R}}
\newcommand{\V}{\mathscr{C}}

\DeclareMathOperator{\psl}{\mbox{PSL}}
\newcommand{\CP}{\C\mbox{\upshape{P}}}
\DeclareMathOperator{\Eq}{Eq}
\newcommand{\Gr}{\mbox{\upshape{Gr}}}
\newcommand{\llangle}{\langle\!\langle}

\newcommand{\rrangle}{\rangle\!\rangle}

\DeclareMathOperator{\LKul}{\Lambda_{Kul}}
\newcommand{\OKul}{\Omega_{Kul}}
\DeclareMathOperator{\pu}{PU}

\newcommand{\GL}{\mbox{\upshape{GL}}}
\DeclareMathOperator{\SL}{SL}
\newcommand{\im}{\mbox{\upshape{im}}}

\DeclareMathOperator{\Ps}{QP}

\newcommand{\Hy}{\mathbb{H}}

\newcommand{\ds}{\displaystyle}

\newcommand{\llP}{[\![}
\newcommand{\rrP}{]\!]}

\DeclareMathOperator{\Isom}{Isom}
\newcommand{\LMyr}{\Lambda_{Myr}}
\title[Comparison of Limit sets]{\sc{Comparison of limit sets for the action of Kleinian groups in $\CP^n.$}}
\begin{document}

\author{Alejandro Ucan-Puc}
\address{CONACYT/Institut de Mathématiques de Jussieu Paris Rive Gauche}
\email{alejandro.ucan-puc@imj-prg.fr}

\author{Jose Seade}
\address{ Instituto de Matemáticas, UNAM}
\email{jseade@im.unam.mx}
%\thanks{  Partially supported by grants of the PAPPIT's project IN101816 and CONACYT's  project 272169} 

%\subjclass{Primary 37F99, 32Q, 32M Secondary 30F40, 20H10, 57M60, 53C}

\maketitle
\begin{abstract}
We compare different notions of limit sets for the action of Kleinian groups on the $n-$dimensional projective space via the irreducible representation $\varrho:\psl(2,\C)\to\psl(n+1,\C).$ In particular, we prove that if the Kleinian group is convex-cocompact, the Myrberg and the Kulkarni limit  coincide.
\end{abstract}

\section*{introduction}
The notion of  \emph{complex Kleinian group} was introduced by Seade and Verjovsky in \cite{Seade2002} as discrete subgroups of  $\psl(n+1,\C),$ $\Gamma$ acting on $\CP^n,$ such that there exists an non-empty open subset of $\CP^n$ where the action of $\Gamma$ is properly-discontinuous. Complex Kleinian groups are examples of really rich holomorphic dynamics, and a way to generalize the Kleinian group theory into a (complex) higher dimensional setting.

Some examples of complex Kleinian groups are \emph{complex hyperbolic groups}, i.e., discrete subgroups of the isometries of the complex hyperbolic space $\Isom(\Hy^n_\C),$ see \cite{Navarrete2006,Cano2017}, for such groups the complex hyperbolic space is part of the open set where the action is properly-discontinuous. Other examples are discrete subgroups of $\psl(3,\C)$ whose action on $\CP^2$ is irreducible, see \cite{Barrera2011}, or when the subgroup is virtually solvable, see \cite{Barrera2021}. Notice that in dimensions bigger that two there no known examples beside the complex hyperbolic case and cyclic subgroups of $\psl(3,\C)$ (see \cite{Cano2017a}), for this reason it is important to provide new examples.

For Kleinian groups the open set of $\CP^1$ where the action is properly discontinuous correspond to the complement of the \emph{limit set,} $\Lambda(\Gamma),$ defined as the closure of accumulation points of the orbits of the group. In the higher-dimensional complex setting, this definition is not sufficient to provide an open subset where the action if properly-discontinuous. For example, if we take the cyclic group of $\psl(3,\C)$ generated by a diagonal matrix $\gamma,$ whose elements are $|\lambda_1|>|\lambda_2|>|\lambda_3|>0,$ $\Lambda(\gamma)$ correspond to three (fixed) points in $\CP^2,$ but the action is not properly discontinuous (see Proposition 3.1.1 in \cite{Cano2013}). Therefore, in higher dimensions, we need to look for other notions of limit sets, for example:
\begin{itemize}
\item The complement of the Equicontinuity region.
\item The Kulkarni limit set (see Definition see \ref{Dfn:KulkarniLimitSet}).
\item The Conze-Guivarc'h limit set (see Definition \ref{Dfn:ConzeGuivarchlimitset}).
\end{itemize}
 
For Kleinian groups of $\psl(2,\C)$ the all previous limit sets coincide with $\Lambda(\Gamma)$ and the associated open set is the maximal open set where the action is properly discontinuous. In higher dimensions, none of the above limit set guarantee that the associated open set of $\CP^n$ is the \emph{maximal} subset where the action is properly-discontinuous and not necessarily all coincide.

The lack of examples of complex Kleinian groups in higher dimensions follows from the fact that the hard computation of the different limit sets and the no precise relations between them. 

One method to provide examples of complex Kleinian groups is via the \emph{irreducible representation} of $\psl(2,\C)$ into $\psl(n+1,\C),$ we call such Kleinian groups \emph{Veronese groups} because their action is intimately related to the Veronese embedding of $\CP^1$ into $\CP^n.$ Naively, we will study the action the well-known Kleinian groups of $\psl(2,\C)$ in $\CP^n.$ The main contribution of this paper relies on the description of the different notion of limit set for some Veronese groups and provide relations between them. In particular, it is proved that for convex co-compact subgroups of $\psl(2,\C),$ the Kulkarni limit set of the corresponding Veronese group coincides with the complement of the region of Equicontinuity, see Theorem \ref{MThm:Kulkarni}.

\begin{Thm}[Main Theorem]\label{MThm:Kulkarni}
Let $\Gamma\subset \psl(2,\C)$ be a Kleinian group, and let $G=\varrho(\Gamma)$ be the associated Veronese group in $\psl(n+1,\C).$ If $\Gamma$ is convex co-compact, then 
\begin{equation}
\LKul(G)=\CP^n\setminus \Eq(G).
\end{equation}
\end{Thm}

The previous theorem provides a generalization of the results known for the cases of complex Kleinian groups in dimension 3 and the complex hyperbolic groups.

This is not the first time that the irreducible action of $\psl(2,\C)$ in $\psl(n+1,\C)$ is studied, we can refer to the \emph{Anosov representations} in \cite{Guichard2012, Dumas2020} where the authors describe domains of discontinuity in flag manifolds. Moreover, the particularity of being Anosov implies the existence of a map from the boundary of the group to the flag manifold, map known as the \emph{limit map}. The techniques used to describe the complement of the Equicontinuity Region and the Kulkarni limit set can be used to describe a the limit maps in the particular case of Veronese groups. As a consequence of Theorem \ref{MThm:Kulkarni}, the Kulkarni limit set is also related to the limit maps of the Anosov representation when the group is convex co-compact. 

In \cite{Conze2000} the authors provide the definition of a limit set using the \emph{proximal} property of maps in $\psl(n+1,\C),$ this limit set is called the \emph{Conze-Guivarc'h limit set} and it is defined for discrete subgroups of $\psl(n+1,\C)$ that contains proximal elements and act strongly irreducible on $\CP^n$ (see Definition \ref{Dfn:ConzeGuivarchlimitset}). The notion of Conze-Guivarc'h limit set extends to the Veronese groups and we can define a new dominant subspaces that generalize the proximality of the elements in the group. We propose a more adequate notion of limit set called the \emph{Extended Conze-Guivarc'h limit set}; such limit set explodes the properties of the irreducible representation and the \emph{KAK-decomposition} of elements of $\psl(2,\C).$ This new notion of limit set allow us to reduce the dimension of the projective subspaces where the orbits accumulates and it has a minimal property as in the Conze-Guivarc'h case.

\begin{Thm}[Main Theorem 2]
 Let $\Gamma\in\psl(2,\C)$ a discrete group that satisfy the (CG)-property, and let $G$ denotes the associated Veronese group. Then $G$ acts properly discontinuous on $\CP^n\setminus L^{Ext}_{CG}(G),$ the complement of the Extended Conze-Guivarc'h limit set. The set $L^{Ext}_{CG}(G)$ is a minimal $G-$invariant closed set where the action is properly discontinuous. 
\end{Thm}
 
%The KAK-decomposition of $\psl(2,\C)$ allows to describe an element $\gamma$ as a product $u\cdot a\cdot v$ where $a$ is a diagonal matrix with $a_{11} >a_{22}>0,$ and $u,v$ are orthogonal matrices $U(n).$ The irreducible representation sends diagonal elements of $\psl(2,\C)$ into diagonal elements of $\psl(n+1,\C),$ and in the particular case of the KAK-decomposition the image of the diagonal matrix has $\lfloor\frac{n+1}{2}\rfloor$ elements bigger than 1.

%A direct consequence from the construction of the extended Conze-Guivarc'h limit set is that if the Kleinian group $\Gamma$ is convex co-compact, then the Kulkarni limit set of $\varrho(\Gamma)$ and the extended Conze-Guivarc'h limit set $L^{Ext}_{CG}(\varrho(\Gamma))$ coincide. 

The paper is organized as follows: in Section \ref{S:Preliminaries} we define the different notions of limit set that are studied on the complex Kleinian group theory, also we provide useful results to they computation. Over Section \ref{S:Irreducible} we describe the $\psl(2,\C)$ action on $\CP^n$ via the irreducible representation, and properties related to complex Kleinian groups. In Section \ref{S:Limitsets} we compute the Myrberg limit set and the Kulkarni limit set of convex co-compact Kleinian groups. Finally, on Section \ref{S:CGlimitset} we introduce the extended Conze-Guivarc'h limit set.

\section{Preliminaries} \label{S:Preliminaries}
\subsection{The Myrberg limit set}

The Myrberg limit set is one of the notions of limit set used in the literature, whose importance relies with its relation with the Equicontinuity region for the complex Kleinian groups. Recall that for a group $G$ acting on a manifold, the \emph{equicontinuity region} $\Eq(G)$ is the open subset of the manifold where the restricted group is a normal family. For discrete subgroups of $\psl(n+1,\C)$ the equicontinuity region is intimately related to the \emph{Quasi-projective transformations set}, introduced by \cite{Furstenberg1963}. 

Let $T:\C^{n+1}\rightarrow\C^{n+1}$ a linear transformation with non-trivial kernel and denote $[\ker T]$ the projectivization of its kernel, the map $T$ induce a map from $\CP^n\setminus[\ker T]$ to $\CP^n,$ denoted by $\llP T\rrP,$ given by
\begin{equation}
\label{S1:Eq1:PseudoProj}
\llP T\rrP ([z])= [T(z)].
\end{equation}

We will call the map $\llP T\rrP$ a \emph{quasi-projective map} and if $M(n+1,\C)$ denotes the set of linear transformations from $\C^{n+1}$ to $\C^{n+1},$ the set $M(n+1,\C)\setminus\{0\}/\C^*$ induces the set of quasi-projective maps and we will denoted by $\Ps(n+1,\C).$

The following proposition describe the relation of the quasi-projective transformations and the equicontinuity set. 

\begin{Prp}(see \cite{Cano2010})
\label{Prop:EquicontinuityGeneral}
Let $(\gamma_m)_{m\in\N}\subset \psl(n+1,\C)$ a sequence of distinct elements then:
\begin{enumerate}[i.]
\item There is a subsequence $(\gamma_{m_j})_{j\in\N}$ and $\gamma_0\in\Ps(n+1,\C)$ such that \[\gamma_{m_j}\xymatrix{\ar[r]_{j\rightarrow\infty}&}\gamma_0\] as points in $\Ps(n+1,\C).$
\item If $(\gamma_{m_j})_{j\in\N}$ is the sequence given by the previous part of this lemma, then $\gamma_{m_j}\xymatrix{\ar[r]_{j\rightarrow\infty}&}\gamma_0,$ as functions, uniformly on compact sets of $\CP^n\setminus[\ker \gamma_0].$
\end{enumerate}
\end{Prp}

\begin{Dfn}\label{Dfn:Myrberglimitset}
Let $\Gamma\subset \psl(n+1,\C),$ and let $\overline{\Gamma}$ denote the set of quasi-projective maps such that there exists a sequence on $\Gamma$ that converges to it. The \emph{Myrberg limit set} of $\Gamma,$ denoted by $\LMyr(\Gamma),$ is the set 
\begin{equation*}
\bigcup_{g\in\overline{\Gamma}} \ker (g).
\end{equation*}
\end{Dfn}

The Proposition \ref{Prop:EquicontinuityGeneral}, implies that for a sequence of different elements in $\psl(n+1,\C),$ the equicontinuous set of the sequence is given by the complement in $\CP^n$ of the kernel of the limit quasi-projective transformation, see \cite{Cano2010} for a detailed description. In particular, the complement of the Equicontinuity region correspond to the Myrberg limit set.

\subsection{The Kulkarni limit set}

The notion of Kulkarni limit set was introduced by R. Kulkarni \cite{Kulkarni1978} for groups acting on manifolds by diffeomorphism. Naively, the Kulkarni limit set removes by phases the undesirable kind of actions (accumulations) in order to have a properly discontinuous action on the complement. 

\begin{Dfn}\label{Dfn:KulkarniLimitSet}
 Let $\Gamma\subset\psl(n+1,\C)$ a discrete subgroup and consider the following subsets of $\CP^n$
 \begin{enumerate}[i.]
  \item $L_0(\Gamma)$ the closure of cluster points of infinite isotropy group.
  \item $L_1(\Gamma)$ the closure of cluster points of orbits of points in $\CP^n\setminus L_0(\Gamma).$
  \item $L_2(\Gamma)$ the closure of cluster points of orbits of compact subsets of $\CP^n\setminus(L_0(\Gamma)\cup L_1(\Gamma)).$
  \end{enumerate}
  The \emph{Kulkarni limit set} for $\Gamma$ is the union $L_0(\Gamma)\cup L_1(\Gamma)\cup L_2(\Gamma)$ and we will denoted by $\LKul(\Gamma).$ We will denote by $\OKul(\Gamma)$ to the complement of the Kulkarni limit set and it will be called, the Kulkarni discontinuity region.
\end{Dfn}

Although the definition of the  Kulkarni limit set guarantees that the action of the group on its complement is properly discontinuous, the computability of this limit set is really complicated and there is no general method to compute it. Even more, there are no precise relations between the $L_j(\cdot),$ see Chapter 7 of \cite{Cano2013} to some atypical relations on the limit set for some groups, also there are no monotony relations in any case.

\begin{Prp}
\label{Ch0:Prp:KulkarniProperties}
 Let $\Gamma$ be a complex Kleinian group. Then:
 \begin{enumerate}[i.]
  \item The set $\LKul(\Gamma)$ is a $\Gamma-$invariant closed set.
  \item The group $\Gamma$ acts properly discontinuously on $\OKul(\Gamma).$
  \item Let $\mathcal{C}\subset\CP^n$ a closed $\Gamma-$invariant set such that for every compact set $K\subset \CP^n\setminus \mathcal{C},$ the set of cluster points of $\Gamma K$ is contained in $(L_0(\gamma)\cup L_1(\Gamma))\cap\mathcal{C}.$ Then $\LKul(\Gamma)\subset \mathcal{C}.$
  \item The equicontinuity set of $\Gamma$ is contained in $\OKul(\Gamma).$
 \end{enumerate}
\end{Prp}

We present two examples of Kulkarni limit set for complex Kleinian groups in the next results, we have to mention that the techniques to prove them are different but in the same spirit the authors use the existence of the Equicontinuity region and its properties. 

\begin{Thm}[Theorem 0.1, \cite{Cano2017}]
Let $\Gamma\subset \pu(n,1)$ be a discrete subgroup; then, the Kulkarni limit set of $\Gamma$ can be described as the hyperplanes tangent to $\partial \Hy^n_\C$ at points in the Chen-Greenberg limit set of $\Gamma$, that is
\begin{equation}
\LKul(\Gamma)=\bigcup_{p\in\Lambda_{CG}(\Gamma)} p^\perp.
\end{equation}
Moreover, if $\Gamma$ is non-elementary, then $\OKul(\Gamma)$ agrees with the equicontinuity set of $\Gamma$ and is the largest open set on which the group acts properly and discontinuously. 
\end{Thm}

\begin{Thm}[Theorem 6.3.6, \cite{Cano2013}]
Let $\Gamma\subset \psl(3,\C)$ be an infinite discrete subgroup without fixed points nor invariant projective lines. Let $\mathcal{E}(\Gamma)$ be the set of all projective lines $\ell$ for which there exists and element $\gamma\in\Gamma$ such that $\ell\subset \LKul(\gamma).$ Then one has:
\begin{enumerate}[i.]
\item Then $\Eq(\Gamma)=\OKul(\Gamma),$ is the maximal open set on which $\Gamma$ acts properly discontinuously. Moreover, if $\mathcal{E}(\Gamma)$ contains more than three projective lines, then every connected component of $\OKul(\Gamma)$ is complete Kobayashi hyperbolic.

\item The set 
\begin{equation}
\LKul(\Gamma)= \overline{\bigcup_{\ell\in\mathcal{E}(\Gamma)}\ell}
\end{equation}
is path-connected. 
\item If $\mathcal{E}(\Gamma)$ contains more than three projective lines, then $\overline{\mathcal{E}(\Gamma)}\subset (\CP^2)^*$ is a perfect set, and it is the minimal closed $\Gamma-$invariant subset of $(\CP^2)^*.$
\end{enumerate}
\end{Thm}

\begin{Rmk}
Notice that the Kulkarni limit set for discrete subgroups of $\psl(n+1,\C)$ contains at least one projective subspace. In particular, for $\psl(3,\C)$ the Kulkarni limit set is made of projective subspace of the same dimension. 
\end{Rmk}

\subsection{The Conze-Guivarc'h limit set}

The Conze-Guivarc'h limit set is inspired by the ideas of \emph{proximal transformations} introduced by Abels, Marguilis and Soifer sin \cite{Abels1995}.

\begin{Dfn}
 We say that a matrix in $\SL(n+1,\C)$ is proximal if it has a maximal norm eigenvalue. In the case of $\psl(n+1,\C)$ we say that an element is proximal if some lift is proximal.
\end{Dfn}

The proximal property in an element $\gamma$ implies that the eigenspace is one dimensional (thinking that $\gamma\in\SL(n+1,\C)$) and its complement is an hyperplane, even more, the transformation acts as a contraction whose limit point is the proper vector space and we will call this vector \emph{dominant vector}.

Recall that every proximal element of $\psl(n+1,\C)$ is loxodromic according to the classification given on \cite{Cano2017c}.

\begin{Dfn}\label{Dfn:ConzeGuivarchlimitset}
Let $\Gamma\subset \psl(n+1,\C)$ that contains proximal elements and whose action on $\CP^n$ is \emph{strongly irreducible,} i.e., there no exist any proper nonzero subspace of $\CP^n$ invariant under the action of a subgroup of finite index in $\Gamma.$ The \emph{Conze-Guivarc'h limit set} for $\Gamma$ is the closure of dominant vectors in $\CP^n.$
\end{Dfn}

We will say that a subgroup of $\psl(n+1,\C)$ satisfy the (CG)-property if the hypothesis of the previous Definition are valid. In \cite{Conze2000} these hypothesis are called (H1) and (H2).

The Conze-Guivarc'h limit set is only applicable to those discrete subsets of $\psl(n+1,\C)$ with proximal elements, but not every discrete subgroup contains proximal elements (see \cite{Barrera2021}). Besides the constrains of the Conze-Guivarc'h limit set, it can be redefined in order to be applicable to all possible discrete subgroups, work done on \cite{Barrera2018}, where the authors extend the Conze-Guivarc'h limit set and proved a duality with the Kulkarni limit set in the case of infinite discrete subgroups of $\psl(3,\C).$

\section{The Veronese Curve in $\CP^n$}\label{S:Irreducible}

Let $\psi_n: \CP^1\to \CP^n$ the map given by
\begin{equation}\label{Eqn:VeroneseMap}
[x:y]\mapsto \left[x^n:\binom{n}{1}x^{n-1}y:\cdots: \binom{n}{n-1}xy^{n-1}:y^n\right].
\end{equation}
The previous is a well defined embedding of $\CP^1$ into $\CP^n$ known as the \emph{Veronese embedding}; we will denote by $\V_n$ to the image of $\CP^1$ under $\psi_n,$ the set $\V_n$ is a normal rational algebraic curve called the \emph{Veronese curve.} The existence of set of points in general position will be necessary in the following.

\begin{Lmm}
 Every subset of $n+1$ different points in $\V_n$ is linearly independent.
\end{Lmm}
\begin{proof}
 Let $\{p_j=\psi([1:t_j])\}_{j=0}^n$ with $t_j\neq t_k$ if $j\neq k.$ Let $\sum_{j=0}^{n}a_jp_j=0$ a linear combination, it follows from the form of the elements that the linear system with variables the ${a_j}'$s has the unique solution $0.$ Therefore $\{p_j\}_{j=0}^n$ is linearly independent.
\end{proof}

\begin{Prp}
 Every subset of $m>n+1$ different points of $\V_n$ is in general position.
\end{Prp}
\begin{proof}
 The proof follows from the previous lemma.
\end{proof}

The Veronese curve is intimately related to the $\SL(2,\C)-$representations into $\SL(n+1,\C).$ Denote by $\varrho_n:\SL(2,\C)\to\SL(n+1,\C)$ the irreducible representation morphism, in abuse of the notation, we will still denote by $\varrho_n$ the group morphism induced by the natural projections of $\SL(\cdot,\C)$ into $\psl(\cdot,\C).$ The following are some of the properties that has the morphism $\varrho_n.$

\begin{Lmm}\label{Lmm:IrreducibleRepresentation}
Given a matrix $A\in\psl(2,\C)$ of the form 
\[A=\begin{bmatrix}
                                                          a & b \\
                                                          c & d
                                                         \end{bmatrix},
\] then the $m\times j$ entry of the matrix $\varrho_n(A)$ is of the form

\begin{equation}
  \label{Eq:MatrixRep}
\sum_{k=\delta_{j,m}}^{\Delta_{j,m}}\binom{n-m}{k}\binom{m}{j-k}a^{n-m-k}c^kb^{m-j+k}d^{j-k}
\end{equation}
where $\delta_{j,m}=\max\{j-m,n\}$ and $\Delta_{j,m}=\min\{j,n-m\}.$
\end{Lmm}

\begin{proof}
It follows from a direct computation, using the action of $\SL(2,\C)$ into the space of homogeneous polinomials of degree n in two complex variables.
\end{proof}
\begin{Prp}
\label{Prop:TypePreserving}
 The morphism $\rho$ is type preserving, \emph{i.e.,} it sends loxodromic, parabolic and elliptic elements of $\psl(2,\C)$ into loxodromic, parabolic and elliptic elements of $\psl(n+1,\C).$ Even more if $\Gamma$ is a discrete subgroup of $\psl(2,\C),$ then $\rho(\Gamma)$ is a discrete subgroup of $\psl(n+1,\C).$
\end{Prp}
\begin{proof}
 For the part of type preserving, it will be sufficient to look at the image of the matrices 
 \[A=\begin{array}{ccc}
    \begin{bmatrix}
     \lambda & 0\\
     0 & \lambda^{-1}
    \end{bmatrix} & \mbox{and} & B=\begin{bmatrix}
    1 & 1\\
    0 & 1
    \end{bmatrix}
   \end{array},
\] a straight computation involving Equation (\ref{Eq:MatrixRep}) implies the claim.

For the second part of the proposition, suppose that $\Gamma<\psl(2,\C)$ a discrete subgroup but $\varrho_n(\Gamma)$ isn't discrete. So there is a sequence $(A_m)_{m\in\N}$ in $\Gamma$ such that $\varrho_n(A_m)\rightarrow Id_{n+1}$ as $m\rightarrow \infty,$ with 
\[A_m=\begin{bmatrix}
a_m & b_m \\
c_m & d_m
\end{bmatrix}.\]

From the expresion of $\varrho_n(A_m),$ we can imply that $A_m\rightarrow Id_2,$ this is a contradiction because $\Gamma$ is discrete. Therefore $\varrho_n(\Gamma)$ is discrete.
\end{proof}

\begin{Prp}
 \label{Prop:VerAutomorphism}
 The group of projective automorphisms of $\V_n$ is $\rho_n(\psl(2,\C)).$
\end{Prp}
\begin{proof}
A straight computation give us that  $\varrho_n(A)\cdot\psi(p)=\psi_n(A\cdot p)\in \V_n,$ for every $A\in\psl(2,\C)$ and $p\in\CP^1.$ We can conclude that $\varrho_n(\psl(2,\C))$ leaves invariant the curve $\V_n.$ Let us suppose that there is an element $B\in \psl(n+1,\C)$ such that $B(\V_n)=\V_n.$ Denote by $\tilde{B}$ the map from $\CP^1$ into $\CP^1$ given by $\tilde{B}([z:w])=\psi_n^{-1}B\psi([z:w]).$ The map $\tilde{B}$ is an holomorphic map, so belongs to $\psl(2,\C);$ even more the following diagram commutes,

\[
\xymatrix{
\CP^1\ar[r]^{\tilde{B}} \ar[d]_{\psi_n} & \CP^1\ar[d]^{\psi_n} \\
 \V_n\ar[r]_B & \V_n }
\]

We can assure that $B|_{\V_n}=\varrho_n(\tilde{B})|_{\V_n}.$ Let us take $n+2$ different points of $\V_n$ in general position, the transformation $B\rho(\tilde{B})^{-1}$ fix the $n+2$ points, this implies that $B\rho(\tilde{B})^{-1}=Id_{n+1}$ in $\CP^n.$ So, the projective automorphisms of $\V$ is $\varrho_n(\psl(2,\C)).$
\end{proof}
\begin{Crl}
 The Veronese embedding $\psi_n$ is $\psl(2,\C)-$equivariant. Even more, for every $\Gamma\subset \psl(2,\C)$ group, $\psi_n$ is $\Gamma-$equivariant.
\end{Crl}

  The previous Corollary implies that the action of subgroups of $\varrho_n(\psl(2,\C))$ on $\V_n$ is essentially the well know action of subgroups of $\psl(2,\C)$ in $\CP^1.$ Even more, if $z\in\CP^1$ belongs to the limit set of $\Gamma,$ then $\psi_n(z)$ is an accumulation point for the orbits of $\varrho_n(\Gamma)$ in $\CP^n.$

\begin{Crl}
\label{Crl:ContentionLimitSets}
 Let $\Gamma$ a discrete subgroup of $\psl(2,\C)$ and $\Lambda(\Gamma)$ its limit set. Then $\psi(\Lambda(\Gamma))$ is contained in $\Lambda(\rho(\Gamma)).$
\end{Crl}

\begin{Dfn}
\label{Dfn:VeroneseGroup}
 A Veronese group is the image under $\varrho_n$ of a discrete subgroup of $\psl(2,\C),$ \emph{i.e.,} the image of a Kleinian group.
\end{Dfn}

\subsection{The KAK-decomposition}

The KAK-decomposition comes from the theory of Lie groups and algebras, it allows us to describe an element of a Lie group as a diagonal matrix up to matrices in a compact subgroup of the Lie group. This decomposition is helpful at the time to describe dynamical aspects of the group and accumulation sets, see \cite{Frances2005}.

Let $G$ be a semi-simple Lie group and $\mathfrak{g}$ its Lie algebra, we will denote by $K$ its maximal compact subgroup and $\mathfrak{k}$ its Lie algebra. It is known that the Killing form induce the following decomposition $\mathfrak{g}=\mathfrak{k}\oplus \mathfrak{k}^\perp.$

Let $\mathfrak{a}$ be the maximal abelian subalgebra on $\mathfrak{k}^\perp,$ this subalgebra induce a decomposition on $\mathfrak{g}$ called the \emph{root decomposition} which is given by 
\[\mathfrak{g}=\bigoplus_{\alpha\in \Sigma\cup\{0\}} \mathfrak{g}_\alpha \qquad  \mathfrak{g}_\alpha =\bigcap_{a\in\mathfrak{a}}\ker\left(ad(a)-\alpha(a)\right),\] where $\Sigma=\left\{\alpha\in \mathfrak{a}^*\setminus\{0\}: \mathfrak{g}_\alpha\neq 0\right\}$ and it is called the restricted root system. Given a total ordering on $\mathfrak{a}^*,$ let say $<,$ we can induce a partition on the root system $\Sigma$ by the sets: $\Sigma^+=\{\alpha\in \Sigma: \alpha>0\}$ and $\Sigma^-=\{\alpha\in \Sigma: \alpha<0\}.$

Using the partition on the root system, we can describe a Weyl chamber on $\mathfrak{a}$ as \[\mathfrak{a}^+=\left\{a\in\alpha:\alpha(a)>0\, \forall \alpha\in\Sigma^+\right\}.\]

\begin{Dfn}
Let $G$ be a Lie group, $K$ its maximal compact subgroup and $\mathfrak{a}^+$ a Weyl chamber (associated to a root decomposition on $\mathfrak{g}$). Any element $g\in G$ can be written as
\begin{equation}\label{Eq:CartanDecomposition}
g=k\exp\left(\mu(g)\right)l,
\end{equation}
where $k,l \in K$ and $\mu:G\to \overline{\mathfrak{a}^+}.$ The map $\mu$ is known as the Cartan projection of $g,$ and by the KAK-decomposition to the product in \eqref{Eq:CartanDecomposition}.
\end{Dfn}

\begin{Rmk}
When $G=\SL(n,\C),$ the KAK-decomposition coincides with the singular value decomposition of the matrix.
\end{Rmk}

The following concepts are related with the Singular value decomposition (KAK-decomposition) in $\SL(n,\C),$ introduced in \cite{Sambarino2019}, that will be helpful to describe the $\lambda-$lemma for the Veronese groups and in future sections a new notion of limit set for Veronese groups. 

\begin{Dfn}
Let $T$ be a linear map between two inner product vector spaces of dimension $n,$ and let \[\sigma_1(T)\geq \cdots \geq\sigma_n(T)\] its \emph{singular values}. If $p\in\{1,\cdots, n-1\}$ and $\sigma_p(T)>\sigma_{p+1}(T),$ then we say that $T$ \emph{has a gap of index p.}
If $T$ has a gap of index $p,$ then we will denote by $U_p(T)$ the $p-$dimensional eigenspace of $\sqrt{T T^*}$ corresponding to the $p$ largest eigenvalues and $S_{n-p}(T)=U_{n-p}(T^{-1}).$
\end{Dfn}

\begin{Rmk}
Notice that the spaces $U_{p}(T)$ and $S_{n-p}(T)$ and the subspaces contained in the flags of Lemma \ref{Prp:Lambdalemma}, are similar. 
\end{Rmk}

\begin{Dfn}
Let $\Gamma$ be a finitely generated group and $\varrho:\Gamma\to \GL(n,\C)$ is \emph{p-dominated} if there exists constants $C,\lambda>0$ such that
\begin{equation}
\frac{\sigma_{p+1}(\varrho(\gamma))}{\sigma_p(\varrho(\gamma))}\leq C e^{\lambda|\gamma| } \qquad \forall \gamma\in \Gamma,
\end{equation}
where $|\cdot|$ denotes the word-length for a given symmetric generating set $S.$
\end{Dfn}

\begin{Rmk}
If a group $\Gamma$ admits a $p-$dominated representation in $\GL(n,\R),$ then $\Gamma$ is word-hyperbolic (see Theorem 3.2 on \cite{Sambarino2019}).
\end{Rmk}

\begin{Dfn}[Lemma 4.7 in \cite{Sambarino2019}]\label{Dfn:LimitMaps}
Let $\varrho:\Gamma\to \GL(n,\C)$ be a $p-$dominated representation. There exists two equivariant continuous map $\xi_p^+:\partial \Gamma\to \Gr(p,d)$ and $\xi_p^-:\partial \Gamma\to \Gr(n-p,n)$ defined by:
\begin{eqnarray}
\xi_p^+(x)&:=&\lim_{m} U_p(\varrho(\gamma_m))\\
\xi_p^-(x)&:=& \lim_{m} S_{n-p}(\varrho(\gamma_m))
\end{eqnarray}
where $(\gamma_m)$ is a quasi-geodesic ray representing $x\in \partial \Gamma.$. 
\end{Dfn}

\begin{Rmk}
In Proposition 4.9 of \cite{Sambarino2019}, the authors prove that if $\varrho:\Gamma\to \GL(n,\C)$ is a $p-$dominated representation, then it satisfies to be a $p-$Anosov representation. This implies that there exists a fiber bundle that splits into sub-bundles that are uniformly contracted by a flow. In the Anosov representation language, the maps $\xi_\pm$ are called the limit maps and satisfies the transversality property
\[\xi_p^+(x)\oplus \xi_p^-(y)=\CP^{n-1}\] where $x,y\in\partial\Gamma$ and $x\neq y.$ 

It is clear from the definition of $\xi_p^\pm$ that their images on $\Gr(p,n)$ and $\Gr(n-p,n)$ are closed $\Gamma-$invariant subsets, \emph{i.e.}, limit set for the action of $\Gamma$ on the corresponding Grassmanian space via the $\varrho-$representation.
\end{Rmk}

The KAK-decomposition on $\SL(n+1,\C)$ let us generalize a dynamical lemma introduced on \cite{Navarrete2006} known as the $\lambda-$lemma. The $\lambda-$lemma describes the accumulation sets on $\CP^2$ for a sequence on a purely loxodromic cyclic subgroup of $\psl(3,\C)$ . This lemma has been generalized for a more general type of (divergent) sequences of $\psl(n+1,\C),$ see Lemma \cite{Cano2017}, where the authors use the KAK-decomposition by bloques. In the particular case of the Veronese groups, we can use the Singular Value Decomposition of divergent sequence in $\psl(2,\C)$ to provide a the KAK decomposition of the image, and describe its accumulation subsets on $\CP^n$ for the action via the irreducible representation.

Before present the $\lambda-$lemma for Veronese subgroups. Notice that if $\gamma\in \psl(2,\C)$ and $D_0=\left(\sigma_1(\gamma),(\sigma_1(\gamma))^{-1}\right)$ is the corresponding diagonal matrix with its singular values, then for $D=\varrho(D_0)$ is the diagonal matrix of singular values of $\varrho(\gamma),$ these singular values satisfy Even more, for any $0<j<n$ we have 
\begin{equation}
\frac{\sigma_{j+1}(\varrho(\gamma))}{\sigma_{j}(\varrho(\gamma))}=\frac{1}{\sigma_1(\gamma)^2}<1
\end{equation}

\begin{Prp}[$\lambda-$Lemma]\label{Prp:Lambdalemma}
Let $(\gamma_m)_{m\in\N}\in\psl(2,\C)$ be a sequence of distinct element such that converges to $\gamma\in \Ps(2,\C)\setminus \psl(2,\C).$ For each $m\in\N,$ let  $\gamma_m=u_ma_mv_m$ be its singular value decomposition. Up to a subsequence, let $u$ and $v$ the limits of $(u_m), (v_m),$ respectively. Then there exists two full flags in $\CP^n,$ let say
 \begin{eqnarray}
  \left(F_1^+,F_{2}^+,\cdots,F_{n}^+\right)&\quad& F_{j}^+=\left[\varrho(u)(\llangle e_1
 ,\cdots,e_j\rrangle)\right]\\
 \left(F_{n+1}^-,F_{n}^-,\cdots, F_2^-\right)&\quad& F_{j}^-=\left[\varrho(v)^{-1}(\llangle e_j,\cdots,e_{n+1}\rrangle)\right]
 \end{eqnarray}
 such that, if $g_m=\varrho(\gamma_m),$  then $g_m\left(F^-_j\setminus F^-_{j+1}\right)$ accumulates on $F^+_{j}$ for each $j=2,\cdots n,$ and $g_m \left(\CP^n\setminus F^-_{2}\right)$ accumulates on $F^+_1.$  Even more, for each $1\leq j\leq n,$ 
 
 \begin{eqnarray*}
 \xi^+_{j}(\gamma)=\lim_{m\to\infty} U_j(g_m)&=&F_{j}^+\\
\xi^-_{j}(\gamma)= \lim_{m\to\infty} S_{n+1-j}(g_m)&=&F_{n+1-j}^-
 \end{eqnarray*}
in the corresponding Grassmanian space and the maps $\xi^\pm_j(\cdot)$ are the ones described in Definition \ref{Dfn:LimitMaps}.
 \label{LambdaLemma}
\end{Prp}

\section{The Notions of Limit set for Veronese Groups}\label{S:Limitsets}

\subsection{The Equicontinuity Region}

As we mention before, the complement of the Equicontinuity region is given by the kernel of quasi-projective limit of sequences. We will use to the KAK-decomposition to describe the kernel and image of the quasi-projective limit for a given sequence, and relate these subsets with the Veronese curve.

Let $\left(\gamma_m\right)_{m\in\N}$ be a sequence of $\psl(2,\C)$ such that $\ds{\lim_{m\to\infty}}\gamma_m=\gamma\in\Ps(2,\C).$ If we take the KAK-decomposition for each element of the sequence \[\gamma_m= u_m a_m v_m, \mbox{ where } \begin{bmatrix} \sigma_1(\gamma_m) & 0 \\ 0 & \sigma_2(\gamma_m)\end{bmatrix}\] 
and $\sigma_1(\gamma_m) > \sigma_2(\gamma_m) $ since we assume that the sequence $\left(\gamma_m\right)_{m\in\N}$ converge to a quasi-projective matrix. Up to a subsequence, we can assume that the both sequence of unitary matrices converge, then the quasi-projective limit of the $(a_m)_{m\in\N}$ is of the form 

\[a=\begin{bmatrix} 1 & 0 \\ 0 & 0\end{bmatrix},\] wich in that case the quasi-projective limit of the $\gamma_m$ sequence is $\gamma=u\cdot a \cdot v.$ Therefore, $\ker(\gamma)=v^{-1}(e_2)$ and $\im(\gamma)=u(e_1).$ Which in certain case both spaces could coincide, for example for the sequence $\gamma_m=g^m$ where $g$ is a parabolic element of $\psl(2,\C).$

\begin{Rmk}\label{Rmk:LimitKernelsImages}
When we pass to $\psl(n+1,\C)$ using the irreducible representation, if $g$ is the quasi-projective limit of the sequence $\varrho(\gamma_m),$ its kernel and image are given by $F^+_1$ and $F^-_{2}$ described in Proposition \ref{Prp:Lambdalemma}. Even more, if we take the sequence $\varrho(\gamma_m^{-1}),$ the quasi-projective limit has by kernel $F^-_{n+1}$ and image $F^+_{n}.$
\end{Rmk}

From the definition of the Veronese curve $\V,$ we have that for the standard basis of $\CP^n$ is valid
\[e_1=\psi\left(\begin{bmatrix}
1 \\ 0 
\end{bmatrix}
\right), \qquad e_{n+1}=\psi\left(\begin{bmatrix}
0   \\ 1
\end{bmatrix}
\right).\] From the equivariant action of $\psl(2,\C)$ on $\V$ we can assure that $F^+_1=\varrho(u)[e_1]\in \V$ and $F^-_{n+1}=\varrho(v^{-1})[e_{n+1}]\in \V,$ even more $F^-_2$ and $F^+_n$ are two hyperplanes that intersects $\V$ only in one point.

\begin{Rmk}
The hyperplanes $F^-_{2}$ and $F^+_{n}$ previously introduced, have a geometric interpretation as \emph{osculating hyperplanes} of $\V$ for a given point. We refer to Chapter 2 of \cite{thesis} for a detailed description. From this interpretation, we can assure that both limit hyperplanes are unique for each quasi-projective limit.
\end{Rmk}

%\begin{Lmm}
%Let $\left(\gamma_m\right)_{m\in\N}$ be a sequence of $\psl(2,\C),$ such that converge to a quasi-projective map $\gamma.$ Consider the sequence $\left(\varrho(\gamma_m)\right)_{m\in\N}$ in $\psl(n+1,\C).$ Then if $g$ is the quasi-projective limit set in $\Ps(n+1,\C),$ then $\im(g)$ is a point in $\V$ and $\ker(g)$ is a hyperplane on $\CP^n$ which is tangent to $\V.$ 
%\end{Lmm}

\begin{Thm}
 Let $\Gamma$ a discrete subgroup of $\psl(2,\C)$ and $G=\rho(\Gamma)$ the correspondent Veronese group, then 
 \begin{equation}
  \label{Eqn:EquicontinuityVer}
  \LMyr(G)=\bigcup_{z\in\Lambda(\Gamma)} T_{\psi(z)}\V.
 \end{equation}
where $\Lambda(\Gamma)$ is the limit set of $\Gamma$ for its action on $\CP^1$ and $T_p\V$ is the unique tangent hyperplane to $\V$ in $p.$
\end{Thm}

\begin{proof}
The theorem follows from the fact that $\Lambda(\Gamma)$ and $\CP^1\setminus \Eq(\Gamma)$ coincide for the action of $\Gamma$ in $\CP^1$ and the Proposition \ref{Prp:Lambdalemma}.
\end{proof}

\subsection{The Kulkarni Limit set}

As we mention earlier, there is no general way to compute the Kulkarni limit set of a group of Kleinian group. From the fact that the Kulkarni limit set is a subset of the complement of the Equicontinuity region ( Proposition \ref{Ch0:Prp:KulkarniProperties}), the complement of the Equicontinuity region is a good candidate for the Kulkarni limit set. 

For the action of a Veronese group it is possible to classify the divergent sequences by its limit flags, as we mention in previous paragraphs. This classification generalize the classification of limit points for Kleinian groups into \emph{conical limit points} and \emph{bounded parabolic limit points}.

\begin{Dfn}[\cite{Kapovich2001}]
Let $\Gamma\subset\psl(2,\C)$ be a discrete subgroup and $z\in\Lambda(\Gamma).$ We say that:
\begin{enumerate}
\item $z$ is a \emph{conical point} if there exists a sequence $\{g_j\in\Gamma\}$ such that for every point $x\in\Lambda(\Gamma)\setminus\{z\},$ the sequence $\{(g_j(x),g_j(z))\}$ is relatively compact in $\Lambda(\Gamma)\times \Lambda(\Gamma)\setminus\mbox{Diagonal}.$

\item $z$ is a \emph{bounded parabolic point} if $stab_\Gamma(z)$ is an abelian subgroup containing parabolic elements and such that $\Lambda(\Gamma)/stab_\Gamma(z)$ is compact.
\end{enumerate}
\end{Dfn}

\begin{Dfn}\label{Dfn:DivergentSequenceType}
Let $\left(\gamma_m\right)_{m\in\N}\subset\psl(2,\C)$ be a sequence that converges to a quasi-projective map $\Ps(2,\C)\setminus\psl(2,\C),$  and $g_m=\varrho(\gamma_m).$ For the sequence $\left(g_m\right)_{m\in \N},$ we will say that:
\begin{enumerate}
\item the sequence is of \emph{loxodromic type} if $\im(g)\not\in\ker(g)$ where $g$ is the quasi-projective limit of the sequence $\left(g_m\right)_{m\in \N}.$

\item the sequence is of \emph{parabolic type} if $\im(g)\in\ker(g)$ where $g$ is the quasi-projective limit of the sequence $\left(g_m\right)_{m\in \N}.$
\end{enumerate}
\end{Dfn}

\begin{Rmk}
We have to mention that the divergent sequence of parabolic type are not exclusive of sequences of  parabolic elements  $\psl(n+1,\C).$ For example, we can found this type of sequences using loxodromic elements in $\psl(2,\C).$ Suppose that 

\[\gamma_1=\left[\begin{array}{cc}
           \lambda & 0\\
           0   & \lambda^{-1}
          \end{array}\right] \: \mbox{and}\: \gamma_2=\left[\begin{array}{cc}
							    a & b \\
							    c & d
                                                         \end{array}\right], \]
                                                         
where both maps are loxodromic and  $\mbox{Fix}(\gamma_2)\subset \C P^1\setminus \{[e_1],\,[e_2]\}.$ If we look for the pseudo-projective limit of $(g_m)_{m\in\N}$ given by

\[g_m=
\left[\begin{array}{cc}
  a & b\lambda^{2m}\\
   c\lambda^{-2m} & d
\end{array}\right]=
\left[\begin{array}{cc}
\frac{a}{\lambda^{2m}} & b \\
\frac{c}{\lambda^{4m}} & \frac{d}{\lambda^{2m}}
\end{array}\right]
\]  whose limit has kernel $[e_1]$ and image $[e_1].$
Even more, we can assure that the sequence $(g_m)_{m\in\N}$ converges to $[e_1]$ in compact sets of $\CP^n\setminus\{[e_1]\},$  and $[g_m(e_1)]$ converges to $[e_1].$ Notice that each element $g_m$ is a loxodromic element, but at the limit have the behavior of a parabolic element sequence.
\end{Rmk}

In the particular case of geometrically finite subgroups of $\psl(2,\C)$ the limit points have dichotomy, \emph{i.e.,} if $\Lambda_c(\Gamma)$ denotes the conical limit points of $\Gamma,$ then $\Lambda(\Gamma)\setminus \Lambda_c(\Gamma)$ is made of bounded parabolic limit points. This property can be translate to Veronese groups as follows.

\begin{Lmm}\label{Lmm:GeometrycallyFiniteVeroneseGroups}
Let $\Gamma\subset\psl(2,\C)$ be a geometrically finite Kleinian group and $G=\varrho(\Gamma)$ the corresponding Veronese group on $\psl(n+1,\C).$ Let $\left(g_m\right)_{m\in\N}\subset G$ be a divergent sequence of distinct elements that converges to a quasi-projective map $g,$ then if the sequence is not of loxodromic type then exists it is of parabolic type. Even more, if the sequence $\left(g_m\right)_{m\in\N}$ is of parabolic type, there exists a parabolic type sequence where all elements are parabolic.
\end{Lmm}
\begin{proof}
The results follows from the definition of conical and bounded parabolic limit points for $\Gamma$ and the fact that a limit point of $\Gamma$ correspond to a divergent sequence of $\Gamma$ that converges to a quasi-projective map $\Ps(2,\C).$

Let $\xi\in\Lambda(\Gamma)$ be a conical limit point. From Theorem 4.39 \cite{Kapovich2001}, for each $\zeta\in\CP^1\setminus\{\xi\},$ there exists a sequence $\left(\gamma_m\right)_{m\in\N}$ of distinct elements such that for $\xi\in \CP^1\setminus \{\zeta\},$ then \[\{\gamma_j(\xi),\gamma_j(\zeta)\}\subset \CP^1\times \CP^1\setminus \mbox{Diagonal}\] is relatively compact subset. 

Le $\gamma=\ds{\lim_{m\to\infty}}\gamma_m$ be the quasi-projective limit set of the previous sequence. If $\ker(\gamma)=\im(\gamma),$ then the set $ \{\gamma_j(\xi),\gamma_j(\zeta)\}$ is no longer relatively compact subset of $\CP^1\times \CP^1\setminus \mbox{Diagonal}$ because $\{\gamma_j(\xi),\gamma_j(\zeta)\}$ accumulates to $(\im(\gamma), \ker(\gamma))\in\mbox{Diagonal}.$ Therefore $\ker(\gamma)\neq\im(\gamma),$ and the sequence $\left(\varrho(\gamma_m)\right) _{m\in\N}$ is a loxodromic type divergent sequence.

If $\xi$ is a bounded parabolic limit point, in particular $stab_\Gamma(\xi)$ is a parabolic subgroup of $\Gamma$ and therefore we can construct a sequence $\{\gamma_m\}_{m\in\N}\subset stab_\Gamma(\xi)$ of distinct elements. For such sequence the quasi-projective limit $\gamma=\ds{\lim_{m\to \infty}\gamma_m}$ satisfy $\im(\gamma)=\ker(\gamma),$ and the corresponding sequence $\left(\varrho(\gamma_m)\right)$ is a parabolic type sequence.
\end{proof}

Note that if $p$ is a bounded parabolic limit point of a Kleinian group $\Gamma,$ by the previous Lemma, there exists a sequence of distinct  on $\varrho(Stab_\Gamma(p))$ that converge to $\psi(p).$ If we use the Proposition \ref{Prp:Lambdalemma} with this sequence we obtain that $F_j^+=F_{n+2-j}^-,$ which means that both flags coincide. But the dynamics of the group $\varrho(Stab_\Gamma(p))$ on $\CP^n$ is like the dynamics of the upper triangular matrices, where the only attractive point is the global fixed point $F_1^+.$ Therefore is not clear that the dynamical behavior of Proposition \ref{Prp:Lambdalemma} is still valid. For this reason, we will work only with convex co-compact subgroups of $\psl(2,\C),$ because their limit points are conical, which implies that for the corresponding Veronese group there are only loxodromic type sequences.

\begin{Thm}
Let $\Gamma\subset \psl(2,\C)$ be a convex co-compact Kleinian group, and $\varrho_n(\Gamma)\subset \psl(n+1,\C)$ be its associated Veronese group. For the action of $\varrho_n(\Gamma)$ in $\CP^n,$ it is satisfied
\begin{equation}
\LKul(\varrho(\Gamma))=\CP^n\setminus\Eq(\varrho(\Gamma))=\bigcup_{z\in\Lambda(\Gamma)}T_{\psi(z)}\V.
\end{equation}
where $T_{\psi(z)}\V$ is the osculating hyperplane to $\V$ at $\psi(z).$
\end{Thm}

\begin{proof}
Recall that from properties of the Kulkarni limit set, we have that $\LKul(\varrho(\Gamma))\subset \CP^n\setminus\Eq(\varrho(\Gamma)).$ We will prove that other contention.

Let $z\in \CP^n\setminus\Eq(\varrho(\Gamma)).$ From the fact that $\LKul(\varrho(\Gamma))\cap\V=\Lambda(\Gamma),$  we can assume that $z\not\in \V.$

From the definition of $\CP^n\setminus\Eq\left(\varrho_n(\Gamma)\right),$ we assure that there exists a limit point $p\in\Lambda(\Gamma)$ such that $z\in T_{\psi(p)}\V.$ If $p$ is a conical limit point, and from the Lemma \ref{Lmm:GeometricallyFiniteVeroneseGroups}, there exists a sequence $(\gamma_j )_{j\in\N} \subset \varrho(\Gamma)$ of distinct elements such that if $g=\ds{\lim_{n\to \infty}}\gamma_j$ then $\im(g)\not\in\ker(g).$ From the $\lambda-$lemma (Proposition \ref{Prp:Lambdalemma}) applied to the sequence $(\gamma_j),$ there exists a pair of full flags 
\begin{eqnarray*}
\left(F^+_1,F^+_{2},\cdots,F^+_{n}\right)&\quad& F_1^+=\im(g)\\
 \left(F^-_{n+1},F^-_{n},\cdots,F^-_{2}\right)&\quad& F^-_2=\ker(g).
\end{eqnarray*}

%\todo{Según yo, poco importa si cambiamos conical por bounded parabolic fixed points, aun más si solo pensamos en tipo de sucesiones}

There exists  $1\leq j\leq n$ such that $z\in F^+_{j},$ from Proposition \ref{Prp:Lambdalemma} we know that there exists a sequence  $(x_m)_{m\in\N}\subset \CP^n$ such that $x_m\to x\in F^-_j\setminus F^-_{j+1}$ and $\gamma_m(x_m)$ accumulates to $z.$

If the sequence $(x_m)_{m\in\N}$ contains infinitely many elements of points of $L_0(\varrho(\Gamma))\cup L_1(\varrho(\Gamma)),$ the limit point $x$ belong to $L_0(\varrho(\Gamma))\cup L_1(\varrho(\Gamma))$ and therefore $z$ too. If the sequence $(x_m)_{m\in\N}$ contains finitely many elements of $L_0(\varrho(\Gamma))\cup L_1(\varrho(\Gamma)),$, after a subsequence, we can assure that $(x_{m_j})_{j\in\N} \subset \CP^n\setminus L_0(\varrho(\Gamma))\cup L_1(\varrho(\Gamma))$ and this implies that $z$ belong to $L_2(\varrho(\Gamma)).$ In any case, $z\in \LKul(\varrho(\Gamma)).$

%If $p$ is a bounded parabolic point, then by Lemma \ref{Lmm:GeometrycallyFiniteVeroneseGroups}, there exists a divergent sequence of purely parabolic elements of $\varrho(Stab_\Gamma(p)).$ Without lose of generality, we can assume that $\psi(p)=[e_1]$ and $Stab_\Gamma(p)$ is a subgroup of the upper triangular matrices of $\sl(2,\C).$ By the Proposition \ref{Prp:Lambdalemma} there exists $j$ such that $z\in F_j^+\setminus F_{j-1}^+.$ Note that, on this case, $F_j\cong \CP^j$ and the $Stab_\Gamma(p)$ restrited to $F_j$ acts as a parabolic subgroup on $F_j,$ therefore for a divergent sequence of $\varrho(Stab_\Gamma(p))$ that converge to $\psi(p),$ its restriction to $F_j$ satisfy Proposition \ref{Prop:EquicontinuityGeneral} and the sequence converge uniformly on compact sets of $F^+_j\setminus [\ker(\gamma_0|_{F^+_j})].$ For a divergent sequence of upper triangular matrices restricted to $F^+_j=\llangle e_1,\cdots e_j\rrangle,$ the kernel of its quasi-projective limit set correspond to $F^+_{j-1},$ which implies that $z\in F^+_j\setminus [\ker(\gamma_0|_{F^+_j})],$ and $z\in L_2(\varrho(\Gamma))$.
\end{proof}

\section{The extended Conze-Guivarc'h limit set}\label{S:CGlimitset}

For subgroups of $\psl(2,\C)$ a proximal element correspond to a loxodromic map and one of its fixed points coincide with the dominant vector of the proximal map. Even more, the loxodromic fixed points are always conical limit points in a Kleinian group, therefore if $\Gamma\subset \psl(2,\C)$ satisfy the (CG)-property, then all its limit points are conical points which implies that $\Gamma$ is \emph{convex-cocompact.}

In sight of the proximal property and the Proposition \ref{Prp:Lambdalemma} we can generalize the notion of proximality for Veronese groups in loxodromic type divergent sequences. Recall that for an element of the $\varrho(\psl(2,\C))$ the diagonal matrix in its KAK-decomposition has $\lfloor \frac{n+2}{2}\rfloor$ ``maximal'' eigenvalues . Therefore, up to unitary matrices, we can assume that all elements are proximal.

\begin{Dfn}
Let $\left(g_m\right)_{m\in\N}$ be sequence of distinct elements in $\varrho(\psl(2,\C)),$ such that converge to some element $g\in \Ps(n+1,\C).$ If $\left(g_m\right)_{m\in\N}$ is a loxodromic type divergent sequence, we define the  of $\left(g_m\right)_{m\in\N}$ to be the projective subspace $F^+_{\lfloor \frac{n+2}{2}\rfloor}=\xi^+_{\lfloor \frac{n+2}{2}\rfloor},$ where $\xi^+_{j}(g)$ is described in Proposition \ref{Prp:Lambdalemma}. 
 \end{Dfn}

\begin{Dfn}
Let $\Gamma$ be a discrete subgroup of $\psl(2,\C)$ and $G=\varrho(\Gamma)$ its correspondent Veronese group. We define the \emph{extended Conze-Guivarc'h limit set} of $G,$ denoted by $L^{Ext}_{CG}(G)$ as
 \begin{equation}
  L^{Ext}_{CG}(\varrho(\Gamma))=\overline{\bigcup_{p\in\Lambda(\Gamma)}\xi^+_{\lfloor \frac{n+2}{2}\rfloor}(\psi(p))}
 \end{equation}
 made of the extended dominated subspaces associated to a divergent sequence that has $\psi(p)$ as the image of its quasi-projective limit.
\end{Dfn}

Clearly the previous set is a closed set whose is invariant under the $\varrho(\Gamma),$ and by its definition it is contained in $\CP^n\setminus\Eq(\varrho(\Gamma)).$

\begin{Prp}
 Let $\Gamma\in\psl(2,\C)$ a discrete group which satisfy the (CG)-property. If $G$ denotes the associated Veronese group, then $\varrho(\Gamma)$ acts properly discontinuous on $\CP^n\setminus L^{Ext}_{CG}(G).$
\end{Prp}
\begin{proof}
From Proposition 2.4 of \cite{Conze2000}, we know that if $\Gamma\subset \psl(2,\C)$ satisfy the (CG)-property the Conze-Guivarc'h limit set of $\Gamma$ is minimal and therefore it coincides with the usual limit set $\Lambda(\Gamma).$ We can assume that for every limit point $p\in\Lambda(\Gamma)$ there exists a proximal map $\gamma$ such that $p$ corresponds to its dominant vector. Recall that in $\psl(2,\C)$ a proximal element correspond to loxodromic element , therefore the sequence $\left(\varrho(\gamma^m)\right)_{m\in\N}$ is a loxodromic type sequence and by Proposition \ref{Prp:Lambdalemma} we have that $\left(\gamma^m\right)_{m\in\N}$ acts properly discontinuous on $\CP^n\setminus \xi^+_{\lfloor \frac{n+2}{2}\rfloor}(p).$
\end{proof}

\begin{Thm}
 Let $\Gamma\in\psl(2,\C)$ a discrete group that satisfy the (CG)-property, let $G$ denotes its associated Veronese group. Then $L^{Ext}_{CG}(G)$ is a minimal $G-$invariant closed set where the action is properly discontinuous.
\end{Thm}

\begin{proof}
Assume that $W\subset \CP^n$ is a closed invariant subset such that the action of $G$ on $\CP^n\setminus W$ is properly discontinuous. From the hypothesis, we have that $\Gamma$ satisfy the (CG)-property, therefore each limit point $p\in\Lambda(\Gamma)$ is conical and there exists a sequence of distinct element $\left(\varrho(\gamma_m)\right)_{m\in\N}\subset G$ that is of loxodromic type. If $\zeta\in W$ in particular $\zeta\in \CP^n\setminus \Eq(G)$ and $\zeta \in\xi^+_n(g)$ where $g=\ds{\lim_{m\to\infty}g_m}$ and $\left(g_m\right)_{m\in\N}$ is a divergent sequence of loxodromic type. If $\zeta\in \xi^+_n(g)\setminus \xi^+_{\lfloor \frac{n+2}{2}\rfloor},$ by the fact that $\left(g_m\right)_{m\in\N}$ is of loxodromic type, then we assure that $\zeta\in \xi^+_{\lfloor \frac{n+2}{2}\rfloor}(f)$ where $f=\ds{\lim_{m\to\infty} g^{-1}_m}\in \Ps(n+1,\C).$
 \end{proof}

The previous Theorem has big implications about our expectations over the Kulkarni limit set. The following Corollary implies that in higher dimensional setting the Kulkarni limit set is not the best option to look forward as a limit set and not just for the higher complication of computing but also for its big composition of elements. 

\begin{Crl}
Let $\Gamma\subset\psl(2,\C)$ be a Kleinian group that satisfy the (CG)-property, and denote by $G$ its associated Veronese. The following are satisfied
 \begin{eqnarray}
  Eq(G)&=&\CP^n\setminus L^{Ext}_{CG}(G)\\
  \OKul(G)& = & \CP^n\setminus L^{Ext}_{CG}(G)
 \end{eqnarray}
\end{Crl}

\bibliographystyle{abbrv}
\bibliography{DVG082021v4}

\end{document}